\documentclass{amsart}

\usepackage{amsmath}
\usepackage{amssymb}
\usepackage{amsthm}
\usepackage{amscd}
\usepackage[latin1]{inputenc}
\usepackage{graphicx}

\newtheorem{teor}{Theorem}
\newtheorem{cor}{Corollary}
\newtheorem{prop}{Proposition}
\newtheorem{lem}{Lemma}
\newtheorem{conj}{Conjecture}
\theoremstyle{definition}

\newtheorem*{rem}{Remark}
\author{Antonio M. Oller-Marc\'{e}n}
\author{Jos\'{e} Mar\'{i}a Grau}
\title{On the base $b$ expansion of the number of trailing zeroes of $b^k!$}

\begin{document}

\begin{abstract}
Let us denote by $Z_{b}(n)$ the number of trailing zeroes in the base b
expansion of $n!$. In this paper we study the connection between
the expression of $\vartheta(b):=\lim_{n\rightarrow \infty}Z_{b}(n)/n$ in base $b$, and that of $Z_{b}(b^{k})$.

In particular, if $b$ is a prime power, we will show the equality
between the $k$ digits of $Z_{b}(b^{k})$ and the first $k$ digits in
the fractional part of $\vartheta (b)$. In the general case we will
see that this equality still holds except for, at most, the last
$\left\lfloor \log _{b}(k)\ +3\right\rfloor $ digits. We finally
show that this bound can be improved if $b$ is square-free and
present some conjectures about this bound.
\end{abstract}
\maketitle
\section{Introduction}
It is a classic topic in Elementary Number Theory to compute the
number of trailing zeroes of the base 10 expansion of the factorial
of an integer (see \cite{GRI}, \cite{GUP} or \cite{NIV} for
instance). In fact, this question can be extended to arbitrary bases
and in \cite{OLL} the behavior of the function $Z_b$ was studied in
detail.

In recent work (see \cite{HAR}, motivated by \cite{SCH}), the following particularity about the number of trailing zeroes in the factorial of the powers of 10 was shown:
\begin{align*}
Z_{10}(10)&=2,\\ Z_{10}(10^{2})&=24,\\ Z_{10}(10^{3})&=249,\\ Z_{10}(10^{4})&=2499,\\ Z_{10}(10^{5})&=24999,\\ Z_{10}(10^{6})&=249998,\\ Z_{10}(10^{7})&=2499999,\\ Z_{10}(10^{8})&=24999999,\\ Z_{10}(10^{9})&=249999998,\\ Z_{10}(10^{10})&=2499999997.
\end{align*}

Although the sequence of final 9's is broken for some values of the
exponent, it seems clear that the number of 9's grows indefinitely.
For instance:
\begin{align*}
Z_{10}(10^{50})=&24999999999999999999999999999999999999999999999989,\\
Z_{10}(10^{100})
=&2499999999999999999999999999999999999999999999999999999999\\
&999999999999999999999999999999999999999982.
\end{align*}

This behavior and the fact that the number of digits of
$Z_{10}(10^m)$ is $m$, motivated the study of the asymptotic
behavior of $\displaystyle{\frac{Z_{10}(10^k)}{10^k}}$ in
\cite{TRE}, where it was shown (as suggested by the previous
examples) that:
\begin{equation*}
\lim_{n\rightarrow\infty}\frac{Z_{10}(n)}{n}=\frac{1}{4}.
\end{equation*}
Also, in \cite{TRE} was suggested (and was proved in \cite{HAR})
that if $b=p_{1}^{r_{1}}\cdots p_{s}^{r_{s}}$, then:
\begin{equation*}
\lim_{n\rightarrow \infty }\frac{Z_{b}(n)}{n}=\min_{1\leq i\leq
s}\frac{1}{r_{i}(p_{i}-1)}.
\end{equation*}

Now, it is natural to wonder about what will be the behavior of the
digits of $Z_b(b^k)$ for other values of $b$. Let us have a look at
the case $k=20$ for various values of $b$:
\begin{align*}
Z_{2}(2^{20})=& 1743392200\\
Z_{3}(3^{20})=& 549755813887\\
Z_{4}(4^{20})=& 23841857910156\\
Z_{5}(5^{20})=& 1828079220031481\\
Z_{6}(6^{20})=& 13298711049602000\\
Z_{7}(7^{20})=& 384307168202282325\\
Z_{8}(8^{20})=& 384307168202282325\\
Z_{9}(9^{20})=& 3039416364764232200\\
Z_{10}(10^{20})=& 24999999999999999996\\
Z_{11}(11^{20})=& 67274999493256000920\\
Z_{12}(12^{20})=& 1916879996223737561074\\
Z_{13}(13^{20})=& 1583746981240066619900\\
Z_{14}(14^{20})=& 13944709237547466926759\\
Z_{15}(15^{20})=& 83131418251991271972652\\
Z_{16}(16^{20})=& 302231454903657293676543\\
Z_{17}(17^{20})=& 254014462915473282650100\\
Z_{18}(18^{20})=& 3187059054099019543609340\\
Z_{19}(19^{20})=& 2088331858752553232964200\\
Z_{20}(20^{20})=& 26214399999999999999999991\\
Z_{21}(21^{20})=& 46369738241158591439532728\\
Z_{30}(30^{20})=&87169610024999999999999999987
\end{align*}

In the light of this data, it may seem that the only interesting
behavior takes place at the multiples of 10. Nevertheless, this is
not the case, as can be seen having a look at the base $b$ expansion
of the considered number:
\begin{align*}
Z_{2}(2^{20})&=\{1,1,1,1,1,1,1,1,1,1,1,1,1,1,1,1,1,1,1,1\}_{2}\\Z_{3}(3^{20})&=\{1,1,1,1,1,1,1,1,1,1,1,1,1,1,1,1,1,1,1,1\}_{3}\\Z_{4}(4^{20})&=\{1,3,3,3,3,3,3,3,3,3,3,3,3,3,3,3,3,3,3,3\}_{4}\\Z_{5}(5^{20})&=\{1,1,1,1,1,1,1,1,1,1,1,1,1,1,1,1,1,1,1,1\}_{5}\\Z_{6}(6^{20})&=\{2,5,5,5,5,5,5,5,5,5,5,5,5,5,5,5,5,5,4,5\}_{6}\\Z_{7}(7^{20})&=\{1,1,1,1,1,1,1,1,1,1,1,1,1,1,1,1,1,1,1,1\}_{7}\\Z_{8}(8^{20})&=\{2,5,2,5,2,5,2,5,2,5,2,5,2,5,2,5,2,5,2,5\}_{8}\\Z_{9}(9^{20})&=\{2,2,2,2,2,2,2,2,2,2,2,2,2,2,2,2,2,2,2,2\}_{9}\\Z_{10}(10^{20})&=\{2,4,9,9,9,9,9,9,9,9,9,9,9,9,9,9,9,9,9,6\}_{10}\\Z_{11}(11^{20})&=\{1,1,1,1,1,1,1,1,1,1,1,1,1,1,1,1,1,1,1,1\}_{11}\\Z_{12}(12^{20})&=\{5,11,11,11,11,11,11,11,11,11,11,11,11,11,11,11,11,11,10,10\}_{12}\\Z_{13}(13^{20})&=\{1,1,1,1,1,1,1,1,1,1,1,1,1,1,1,1,1,1,1,1\}_{13}\\Z_{14}(14^{20})&=\{2,4,9,4,9,4,9,4,9,4,9,4,9,4,9,4,9,4,9,1\}_{14}\\Z_{15}(15^{20})&=\{3,11,3,11,3,11,3,11,3,11,3,11,3,11,3,11,3,11,3,7\}_{15}\\Z_{16}(16^{20})&=\{3,15,15,15,15,15,15,15,15,15,15,15,15,15,15,15,15,15,15,15\}_{16}\\Z_{17}(17^{20})&=\{1,1,1,1,1,1,1,1,1,1,1,1,1,1,1,1,1,1,1,1\}_{17}\\Z_{18}(18^{20})&=\{4,8,17,17,17,17,17,17,17,17,17,17,17,17,17,17,17,17,17,14\}_{18}\\Z_{19}(19^{20})&=\{1,1,1,1,1,1,1,1,1,1,1,1,1,1,1,1,1,1,1,1\}_{19}\\Z_{20}(20^{20})&=\{4,19,19,19,19,19,19,19,19,19,19,19,19,19,19,19,19,19,19,11\}_{20}\\Z_{21}(21^{20})&=\{3,10,10,10,10,10,10,10,10,10,10,10,10,10,10,10,10,10,10,5\}_{21}\\Z_{22}(22^{20})&=\{2,4,8,17,13,4,8,17,13,4,8,17,13,4,8,17,13,4,8,14
\}_{22} \\Z_{23}(23^{20})&=\{1,1,1,1,1,1,1,1,1,1,1,1,1,1,1,1,1,1,1,1\}_{23}\\Z_{24}(24^{20})&=\{7,23,23,23,23,23,23,23,23,23,23,23,23,23,23,23,23,23,23,18\}_{24}\\Z_{25}(25^{20})&=\{3,3,3,3,3,3,3,3,3,3,3,3,3,3,3,3,3,3,3,3\}_{25}
\end{align*}

Of course, the origin of these cyclic repetitions is closely related
to the periodic expression of
$\vartheta(b)=\displaystyle{\lim_{n\rightarrow\infty}\frac{Z_b(n)}{n}}$
when expressed in base $b$. In this paper we study the relation
between $\vartheta(b)$ and the base $b$ expansion of $Z_b(b^k)$. If
$b$ is a prime power, we will show the equality between the $k$
digits of $Z_b(b^k)$ and the first $k$ digits in the fractional part
of $\vartheta(b)$ (i. e., $\lfloor\vartheta(b)b^k\rfloor$). In the
general case this equality does not hold. We are then interested in
investigating the number of final digits of $Z_b(b^k)$ which break
the previous coincidence. To do so, let us consider
$\eta(b,k):=\lfloor\log_b(\lfloor\vartheta(b)b^k\rfloor-Z_b(b^k))+1\rfloor$.
This represents the number of digits of the base $b$ expansion of
$\lfloor\vartheta(b)b^k\rfloor-Z_b(b^k)$ and observe that the number
of unequal final digits between $\lfloor\vartheta(b)b^k\rfloor$ and
$Z_b(b^k)$ is at most $\eta(b,k)+1$. We will show that
$$\eta(b,k)\leq\lfloor\log_b k+2\rfloor$$
is the best possible upper bound for $\eta(b,k)$. Nevertheless, we
will improve this upper bound in the case when $b$ is square-free
and present some conjectures about this bound.

The paper is organized as follows, in the second section we present
the basic facts and technical results that will be used in the rest
of the paper. In the third section we study the prime-power case,
establishing the equality between the $k$ digits of $Z_b(b^k)$ and
the first $k$ digits of the base $b$ expansion of $\vartheta(b)$.
Finally, in the fourth section, we study the general case, where the
latter equality does not hold and give a bound to the number of
unequal digits.

\section{Technical results}
In this section we present some technical results which will be very
useful in the sequel. Some of them are well-known and are presented
without proof. The section is divided into three parts, the first is
devoted to results related to $Z_b(b^k)$, the second mostly deals
with the base $b$ expansion of $\vartheta(b)$ when $b$ is a
prime-power and the third one is devoted to present the main lemma
which will be crucial in the paper.

\subsection{Some results about $Z_b(n)$}
We start this subsection with the following well-known lemma, which
was first proved by Legendre (see \cite{LEG}), that we present
without proof.

\begin{lem}
\
\begin{enumerate}
\item $\displaystyle{Z_p(n)=\sum_{i\geq1}\left\lfloor\frac{n}{p^i}\right\rfloor=\frac{n-\sigma_p(n)}{p-1}}$, where $\sigma_p(n)$ is the sum of the digits of the base $p$ expansion of $n$.
\item $Z_{p^r}(n)=\left\lfloor\displaystyle{\frac{Z_p(n)}{r}}\right\rfloor$ for every $r\geq1$.
\item If $b=p_1^{r_1}\cdots p_s^{r_s}$, then $Z_b(n)=\displaystyle{\min_{1\leq i\leq s}} Z_{p_i^{r_i}}(n)$.
\end{enumerate}
\end{lem}

As a straightforward consequence of the preceding lemma we obtain:

\begin{cor}
Let $l$ be any integer and $p$ be a prime. Then:
$$Z_p(lp^n)=lZ_p(p^n)+Z_p(l).$$
\end{cor}

Recall that $\vartheta(b)=\displaystyle{\lim_{n\rightarrow\infty}\frac{Z_b(n)}{n}}$. In \cite{HAR} and \cite{TRE} an explicit expression for $\vartheta(b)$ was given. Namely:

\begin{prop}
If $b=p_1^{r_1}\cdots p_s^{r_s}$, then:
$$\vartheta(b)=\min_{1\leq i\leq s}\frac{1}{r_i(p_i-1)}.$$
\end{prop}

\begin{rem}
The sequence $S_n=\dfrac{1}{\vartheta(n)}$ appears in \emph{The
On-Line Encyclopedia of Integer Sequences} as sequence A090624. It
is interesting to observe that it was included in the Encyclopedia 5
years before the formula for $\vartheta(n)$ was found.
\end{rem}

The following lemma will be of great importance in sections 3 and 4.

\begin{lem}
Let $k\geq 0$ be an integer.
\begin{enumerate}
\item If $b>1$ is an integer, then: $$0<Z_b(b^{k+1})-bZ_b(b^k).$$
\item If $b$ is a prime power, then: $$0<Z_b(b^{k+1})-bZ_b(b^k)<b.$$
\end{enumerate}
\end{lem}
\begin{proof}
\begin{enumerate}
\item Fist of all observe that $k\lfloor x\rfloor\leq \lfloor kx\rfloor$ for all $k\in\mathbb{Z}$, $x\in\mathbb{R}$ and that if $x\notin\mathbb{Z}$ and $kx\in\mathbb{Z}$ the inequality is strict. Now, for some prime divisor of $b$ (with exponent $r$ in the decomposition of $b$) we have that: $$bZ_b(b^{k})=b\left\lfloor\frac{1}{r}\displaystyle{\sum_{i\geq 1}\left\lfloor\frac{b^{k}}{p^i}\right\rfloor}\right\rfloor<\left\lfloor\frac{1}{r}\sum_{i\geq 1}\left\lfloor\frac{b^{k+1}}{p^i}\right\rfloor\right\rfloor=Z_b(b^{k+1}),$$
since $\displaystyle{\frac{b^{k}}{p^{r(k+1)}}}$ is not an integer,
while $\displaystyle{b\frac{b^{k}}{p^{r(k+1)}}}$ is.
\item Put $b=p^n$. Then, Corollary 1 implies that $Z_p(p^{(k+1)n})=p^nZ_p(p^{kn})+Z_p(p^n)$.

Now, if $r$ is the reminder of the division between $Z_p(p^{(k+1)n})$ and $n$ and $s$ is the reminder of the division between $Z_p(p^{kn})$ and $n$ it follows that:
$$Z_{p^n}(p^{(k+1)n})=\frac{Z_p(p^{(k+1)n})-r}{n},$$
$$Z_{p^n}(p^{kn})=\frac{Z_p(p^{kn})-s}{n}.$$

Thus, since $0\leq r,s\leq n-1$:
\begin{align*}
Z_{p^n}(p^{(k+1)n})-p^nZ_{p^n}(p^{kn})&=\frac{Z_p(p^n)+p^ns-r}{n}\leq\frac{Z_p(p^n)+p^ns}{n}=\\&=\frac{\frac{p^n-1}{p-1}+p^ns}{n}<\frac{p^n+p^ns}{n}\leq\\&\leq\frac{p^n+p^n(n-1)}{n}=p^n.
\end{align*}
\end{enumerate}
\end{proof}

\subsection{The base $p^n$ expansion of $\vartheta(p^n)$}
Let us start by introducing some notation. With
$q=\{d.d_{1}d_{2}...d_{t}\overbrace{d_{t+1}...d_{t+s}}\}_{b}$ we
mean that the fractional part of $q$ in base $b$ consists of $t$
digits followed by a periodic sequence of $s$ digits
($d_{t+i}=d_{t+i+s}$ for all $i>0$). Clearly $t$ can be arbitrarily
large and the length of the period can be any multiple of $s$, so we
will usually assume that $t$ and $s$ are minimal. We will say that
$q$ is exact in base $b$ if there exists $k\geq 1$ such that $d_i=0$
for every $i\geq k$ or $d_i=b-1$ for every $i\geq k$; i.e.:
$$q=\{d.d_{1}d_{2}...d_{t}\overbrace{b-1}\}_{b}=\{d.d_{1}d_{2}...d_{t}+1%
\overbrace{0}\}_{b}$$

\begin{lem}
Let $p$ be a prime and $1\leq r\in\mathbb{Z}$. Then $\dfrac{1}{r(p-1)}$ is exact in base $p^r$ if and only if $p=2$ and $r$ is a power of 2.
\end{lem}
\begin{proof}
This is a straightforward consequence from the fact that $\dfrac{1}{n}$ is exact in base $p^r$ if and only if $\textrm{rad}(n)=p$; i. e., if and only if $n$ is a power of $p$.
\end{proof}

Now we will present some results about the base $b$ expansion of $\vartheta(b)$ when $b=p^n$ is a prime-power.

\begin{lem}
Let $p$ be a prime and $b=p^n$ with $n\in\mathbb{N}$. Then:
$$\vartheta (p^{n})=\{0.d_{1}d_{2}...d_{t}\overbrace{d_{t+1}...d_{t+s}}\}_{b}\Longleftrightarrow \frac{p^{nt}}{n}\sum_{i=0}^{sn-1}p^{i}\in
\mathbb{Z}.$$
\end{lem}
\begin{proof}
First of all observe that:
$$\frac{p^{nt}}{n}\sum_{i=0}^{sn-1}p^{i}=\frac{p^{nt}(p^{sn}-1)}{n(p-1)}=b^{t}\vartheta (p^{n})(b^{s}-1).$$

Let us suppose that $\vartheta
(p^{n})=\{0.d_{1}d_{2}...d_{t}\overbrace{d_{t+1}...d_{t+s}} \}_{b}$.
Then we have that:
\begin{align*}
b^{t}\vartheta(p^{n})=&\{d_{1}d_{2}...d_{t}.\overbrace{d_{t+1}...d_{t+s}}\}_{b},\\b^{t+s}\vartheta (p^{n})=&\{d_{1}d_{2}...d_{t}d_{t+1}...d_{t+s}.\overbrace{
d_{t+1}...d_{t+s}}\}_{b}.
\end{align*}
and it is enough to substract both expressions to obtain that
$b^{t}\vartheta (p^{n})(b^{s}-1)\in \mathbb{Z}$.

Conversely, assume that $b^{t}\vartheta
(p^{n})(b^{s}-1)\in\mathbb{Z}$. It is easy to see that there exists
a sequence $\{d_{i}\}_{i\in \mathbb{N}}\subset \mathbb{N}$, not
eventually null, such that:
$$\vartheta (p^{n})=\sum_{i=1}^{\infty }d_{i}p^{-ni}\ \textrm{with}\ 0\leq
d_{i}<b\text{, }\forall i\in \mathbb{N}.$$

Consequently:
\begin{align*}
b^{t}\vartheta (p^{n})=&z_{1}+\sum_{i=1}^{\infty }d_{t+i}p^{-ni}\ \textrm{with}\ z_{1}\in \mathbb{Z},\\b^{(t+s)}\vartheta (p^{n})=&z_{2}+\sum_{i=1}^{\infty }d_{t+s+i}p^{-ni}\ \textrm{with}\ z_{2}\in \mathbb{Z},
\end{align*}
where
$$0<\sum_{i=1}^{\infty }d_{t+i}p^{-ni}\leq 1,$$
$$0<\sum_{i=1}^{\infty }d_{t+s+i}p^{-ni}\leq 1.$$

Since $b^{(t+s)}\vartheta (p^{n})-b^{t}\vartheta (p^{n})=b^{t}\vartheta
(p^{n})(b^{s}-1)\in \mathbb{Z}$ we have that:
$$\sum_{i=1}^{\infty }d_{t+s+i}p^{-ni}=\sum_{i=1}^{\infty
}d_{t+i}p^{-ni}.$$

From this fact it readily follows that $d_{t+i}=d_{t+s+i}$ and the proof is complete.
\end{proof}

\begin{prop}
Let $p$ be a prime, $b=p^{n}$ and $(n,s,k)\in \mathbb{N}^{3}$. Then:
$$\frac{b^{k}}{n}\sum_{i=0}^{sn-1}p^{i}\in\mathbb{Z}\Longleftrightarrow\frac{b}{n}\sum_{i=0}^{sn-1}p^{i}\in \mathbb{Z}.$$
Or, in other words:
$$\frac{p^{nk}(p^{sn}-1)}{n(p-1)}\in\mathbb{Z}\Longleftrightarrow\frac{p^{n}(p^{sn}-1)}{n(p-1)}\in \mathbb{Z}.$$
\end{prop}
\begin{proof}
We can write $n=p^rn'$ with $r\geq 0$ and $\gcd(p,n')=1$. Then $b^k=p^{kn}=p^{p^rn'k}$ and observe that $r<p^r<p^rn'<p^rn'k$.

Assume that $\displaystyle{\frac{b^{k}}{n}\sum_{i=0}^{sn-1}p^{i}\in\mathbb{Z}}$, then $\displaystyle{\frac{p^{p^rn'k-r}}{n'}\sum_{i=0}^{sn-1}p^{i}\in\mathbb{Z}}$ and, since $\gcd(p,n')=1$ it follows that $\displaystyle{\frac{1}{n'}\sum_{i=0}^{sn-1}p^{i}\in\mathbb{Z}}$ and $\displaystyle{\frac{p^{p^rn'-r}}{n'}\sum_{i=0}^{sn-1}p^{i}\in\mathbb{Z}}$. But $\displaystyle{\frac{p^{p^rn'-r}}{n'}=\frac{b}{n}}$ and we are done.

The converse is obvious.
\end{proof}

\begin{cor}
The base $p^n$ expansion of $\vartheta(p^n)$  is pure periodic or mixed periodic with only one non-periodic figure; i. e., either
$\vartheta (p^{n})=\{0.d_{1}\overbrace{d_{2}...d_{s+1}}\}_{p^{n}}$ with $d_{s+1}\neq d_{1}$
 or $\vartheta (p^{n})=\{0.\overbrace{d_{1}...d_{s}}\}_{p^{n}}$
\end{cor}
\begin{proof}
If $\vartheta (p^{n}) =\{0.d_{1}d_{2}...d_{t}\overbrace{d_{t+1}...d_{t+s}}
\}_{p^{n}}$, then we have that  $\displaystyle{\frac{p^{nt}}{n}\sum_{i=0}^{sn-1}p^{i}\in
\mathbb{Z}}$. By the previous proposition this implies that
$\displaystyle{\frac{p^{n}}{n}\sum_{i=0}^{sn-1}p^{i} \in \mathbb{Z}}$
 and, consequently, $\vartheta (p^{n})=\{0.d_{1}\overbrace{d_{2}...d_{s+1}}
\}_{p^{n}}$. Finally, if $d_{s+1}=d_{1}$, then $\vartheta (p^{n})$
is pure periodic and this completes the proof.
\end{proof}

\subsection{The main lemma}
The following lemma will be crucial in the next section.

\begin{lem}
Let $\left\{ S_{n}\right\} _{n\in \mathbb{N}}$ be a sequence of integers and define $\Delta_{1}:=S_{1}$ and $\Delta_{n}:=S_{n+1}-bS_{n}$. If $0<\Delta _{n}<b$, for all $n\in \mathbb{N}$, then the following hold:
\begin{enumerate}
\item $\displaystyle{S_{n}=\sum_{i=1}^{n}b^{n-i}\Delta _{i}}$.
\item $\left\lfloor \log _{b}S_{n}\right\rfloor +1=n$.
\item $\displaystyle{\ell :=\lim_{n\rightarrow \infty}\dfrac{S_{n}}{b^{n}}=\sum_{i=1}^{\infty
}b^{-i}\Delta _{i}}$.
\item If $\Delta _{i}$ $=(b-1)$ for all $i>1$, then $S_{k}=\ell b^{k}-1$
for every $k>1$.
\item If $\ell$ is not exact in base $b$, then $S_{k}=\left\lfloor \ell b^{k}\right\rfloor $ for every $k$.
\item If $\ell =\sum_{i=1}^{\infty }b^{-i}\Delta _{n}\in \mathbb{Q}$, then there exists $(t,s)\in \mathbb{N}^{2}$ with $s>0$ such that $\Delta
_{n+s}=\Delta _{n}$ (and $S_{n+s}=S_{n})$ for all $n>t$.
\end{enumerate}
\end{lem}
\begin{proof}
\begin{enumerate}
\item It follows from inductive arguments, since $S_1=\Delta_1$.
\item Consequence of (1).
\item Observe that $\displaystyle{\frac{S_{n}}{b^{n}}=\sum_{i=1}^{n}b^{-i}\Delta _{i}}$ and take limits.
\item We must consider two cases.

If $\Delta _{1}<(b-1)$, then $\ell=\frac{(\Delta _{1}+1)}{b}$, $\ell
b^{k}-1=(\Delta _{1}+1)b^{k-1}-1$ and also:
$$S_{k}=\sum_{i=1}^{k}b^{k-i}\Delta _{i}=\Delta
_{1}b^{k-1}+\sum_{i=2}^{k}b^{k-i}(b-1)=\Delta _{1}b^{k-1}+b^{k-1}-1.$$

Now, if $\Delta _{1}=b-1$, then $\ell =1$ and:
$$S_{k}=\sum_{i=1}^{k}b^{k-i}(b-1)=\sum_{i=1}^{k}b^{k+1-i}-\sum_{i=1}^{k}b^{k-i}=b^{k}-1.$$
\item Observe that
$$\ell b^{k}=\sum_{i=1}^{\infty }b^{k-i}\Delta
_{i}=\sum_{i=1}^{k}b^{k-i}\Delta _{i}+\sum_{i=1}^{\infty
}b^{-i}\Delta _{i+k}.$$ Now, if $\ell$ is not exact, it follows that
$\displaystyle{\sum_{i=1}^{\infty }b^{-i}\Delta _{i+k}<1}$ and
consequently:
$$\left\lfloor \ell b^{k}\right\rfloor =\left\lfloor
\sum_{i=1}^{k}b^{k-i}\Delta _{i}+\sum_{i=1}^{\infty }b^{-i}\Delta
_{i}\right\rfloor =\sum_{i=1}^{k}b^{k-i}\Delta _{i}=S_{k}.$$
\item It is clear since the base $b$ expansion of any rational number is periodic.
\end{enumerate}
\end{proof}

\section{The prime power case}
The next theorem establishes the equality between the digits of the
base $b$ expansion of $Z_b(b^k)$ and the first $k$ digits of the
base $b$ expansion of $\vartheta(b)$ if $b$ is a prime power. In
passing we will also prove some other interesting properties.

\begin{teor}
Let $p$ be a prime and $b=p^n$. Consider the sequence $a_{1}=Z_{b}(b)$, $a_{k}:=Z_{b}(b^{(k+1)})-bZ_{b}(b^{k})$ and let $s$ be the smallest integer such that $\displaystyle{\theta :=\frac{b}{n}\sum_{i=0}^{sn-1}p^{i}\in \mathbb{Z}}$.
Then, the following hold:
\begin{enumerate}
\item $\displaystyle{Z_{b}(b^{k})=\sum_{i=1}^{k}a_{i}b^{k-i}}$.
\item $\displaystyle{\vartheta (b)=\lim_{n\rightarrow \infty }\frac{Z_{b}(b^{n})}{b^{n}}
=\sum_{i=1}^{\infty }b^{-i}a_{i}}$.
\item The base $b$ expansion of $\vartheta(b)$ is:
\begin{equation*}
\vartheta (b)=\left\{
\begin{array}{ccc}
\{0.a_{1}\overbrace{a_{2}\cdots a_{s+1}}\}_{b}\text{ } & if &
\dfrac{
\theta }{b}\notin \mathbb{Z}. \\
\{0.\overbrace{a_{1}a_{2}\cdots a_{s}}\}_{b} & if & \dfrac{\theta
}{b}\in \mathbb{Z}.
\end{array}
\right.
\end{equation*}
\item $a_{k}=a_{k+s}$ for all $k>1$. Moreover, if $\dfrac{\theta }{b}\in
\mathbb{Z}$, then $a_{k}=a_{k+s}$ for all $k>0$.
\item \begin{equation*}
\#\{a_{k}\}_{i\in \mathbb{N}}=\left\{
\begin{array}{c}
s,\text{ if }\dfrac{\theta }{b}\in \mathbb{Z}. \\
s+1,\text{ otherwise.}
\end{array}
\right.
\end{equation*}
\item If $\dfrac{\theta }{b}\in \mathbb{Z}$, then $Z_{b}(b^{s})=\dfrac{\theta }{b}$. Otherwise, $Z_{b}(b^{s+1})=\theta +Z_{b}(b)$.
\item If $b+1$ is not a Fermat number, then $Z_{b}(b^{k})=\left\lfloor b^{k}\vartheta (b)\right\rfloor$.
\item If $b+1$ is a Fermat number, then $Z_{b}(b^{k})=b^{k}\vartheta (b)-1$.
\end{enumerate}
\end{teor}
\begin{proof}
First of all observe that, due to Corollary 2, there exists an
integer $s$ such that $\displaystyle{\theta
=\frac{b}{n}\sum_{i=0}^{sn-1}p^{i}\in \mathbb{Z}}$ or, equivalently,
such that $\vartheta
(p^{n})=\{0.d_{1}\overbrace{d_{2}...d_{s+1}}\}_{b}$. Moreover, if
$\displaystyle{\frac{1}{n}\sum_{i=0}^{sn-1}p^{i}\in \mathbb{Z}}$,
then $\vartheta (p^{n})=\{0.\overbrace{d_{1}...d_{s}}\}_{b}$. Also,
by Lemma 2, we have that $0<a_k<b$ so we are in the conditions of
Lemma 5.

After these considerations we can proceed to the proof of the theorem.
\begin{enumerate}
\item Apply Lemma 5 (1).
\item Apply Lemma 5 (3).
\item Due to Corollary 2.
\item Idem.
\item Obvious by the minimality of $s$.
\item If $\dfrac{\theta }{b}\in \mathbb{Z}$, then $\displaystyle{\frac{1}{n}\sum_{i=0}^{sn-1}p^{i}=\frac{\theta }{b}\in \mathbb{Z}}$. This implies that $\vartheta (b)=\{0.\overbrace{a_{1}a_{2}\cdots a_{s}}\}_{b}$, so $b^{s}\vartheta (b)=\{a_{1}a_{2}\cdots a_{s}.\overbrace{a_{1}a_{2}\cdots a_{s}
}\}_{b}$ and, consequently,
$$\displaystyle{\frac{\theta }{b}=\vartheta
(b)(b^{s}-1)=\sum_{i=1}^{s}a_{i}b^{k-i}=Z_{b}(b^{s})}.$$

Now, if $\dfrac{\theta }{b}\notin \mathbb{Z}$ then
$\displaystyle{\frac{b}{n}\sum_{i=0}^{sn-1}p^{i}=\theta \in
\mathbb{Z}}$ so $\vartheta (b)=\{0.a_{1}\overbrace{a_{2}a_{3\cdots
}a_{s+1}} \}_{b}$ and it follows that:
\begin{align*}
b^{s+1}\vartheta (b)=&\{a_{1}a_{2\cdots }a_{s+1}.\overbrace{
a_{2}a_{3\cdots }a_{s+1}}\}_{b},\\
b\vartheta (b)=&\{a_{1}.\overbrace{a_{2}a_{3\cdots }a_{s+1}}\}_{b}.
\end{align*}
Consequently,
$$\displaystyle{\theta =\frac{b}{n}\sum_{i=0}^{sn-1}p^{i}=b\vartheta
(b)(b^{s}-1)=\sum_{i=1}^{s+1}a_{i}b^{k-i}-a_{1}=Z_{b}(b^{s+1})-Z_{b}(b)}.$$
\item If $p^n+1$ is not a Fermat number, then $\vartheta(p^n)$ is not exact due to Lemma 4. Then, it is enough to apply Lemma 5 (5).
\item If $p^n+1$ is a Fermat number; then $p=2$ and $n$ is a power of 2. In this case $\vartheta(p^n)$ is exact and Lemma 5 (4) applies.
\end{enumerate}
\end{proof}

Let us recall that a base-$b$ repunit with $k$ digits, $R_k^{(b)}$, is an integer whose base $b$ expansion consists exactly of $k$ ones; i. e.:
$$R_k^{(b)}:=\{1,1,\overset{k)}{\ldots },1\}_{b}=\sum_{i=0}^{k-1}b^{i}=\frac{b^k-1}{b-1}.$$
In the same way, a base-$b$ repdigit with $k$ digits is a multiple
of a base-$b$ repunit with $k$ digits, i. e., an integer of the form
$\alpha R_k^{(b)}=\{\alpha ,\alpha ,\overset{k)}{\ldots },\alpha
\}_{b}$ with $1\leq\alpha\leq b-1$.

\begin{prop}
Let $p$ be a prime. If $\displaystyle{\frac{1}{n}\sum_{i=0}^{n-1}p^{i}=\frac{R_{n}^{(p)}}{n}\in\mathbb{Z}}$, then
$Z_{p^{n}}(p^{nk})$ is a base-$p^{n}$ repdigit with $k$ digits for all $k\in \mathbb{Z}$. Namely:
$$Z_{p^{n}}(p^{nk})=\frac{R_{n}^{(p)}R_{k}^{(p^{n})}}{n}.$$
\end{prop}
\begin{proof}
Theorem 1 implies that $\vartheta (p^{n})=\{0.\overbrace{a_{1}}\}_{p^{n}}$, where $a_{1}=\displaystyle{\frac{R_{n}^{(p)}}{n}}$. Consequently,
$$Z_{p^{n}}(p^{nk})=\{a_{1},a_{1},\overset{k)}\ldots
,a_{1}\}_{_{p^{n}}}=\frac{R_{n}^{(p)}}{n}R_{k}^{(p^{n})}.$$
\end{proof}

It is interesting to particularize the previous result for $n=1,2$.

\begin{cor}
Let $p$ be a prime, then $Z_p(p^k)$ is a base-$p$ repunit with $k$ digits for every integer $k$.
\end{cor}
\begin{proof}
Follows immediately from the previous proposition, since $\displaystyle{\frac{R_1^{(p)}}{1}=1\in\mathbb{Z]}}$.
\end{proof}

\begin{cor}
Let $p$ be an odd prime, then $Z_{p^2}(p^{2k})$ is a base-$p^2$ repdigit.
\end{cor}
\begin{proof}
If $p$ is odd, then $\dfrac{R_2^{(p)}}{2}=\dfrac{p+1}{2}\in\mathbb{Z}$.
\end{proof}

If $p$ is odd, the above corollary can be generalized for any power of 2. Namely, we have the following.

\begin{prop}
If $p$ is an odd prime, then $Z_{p^{2^m}}(p^{2^mk})$ is a base-$p^{2^mk}$ repdigit.
\end{prop}
\begin{proof}
$\displaystyle{\frac{R_{2^m}^{(p)}}{2^m}=\frac{1}{2^m}\prod_{i=0}^{m-1}(1+p^{2^i})\in\mathbb{Z}}$.
\end{proof}

\begin{rem}
We have seen that $\vartheta (p)=\{0.\overbrace{1}\}_{p}$ for every
prime $p$. Nevertheless, the set of pairs $(b_1,b_2)\in\mathbb{Z}^2$
such that the base $b_1$ expansion of $\vartheta(b_1)$ and the base
$b_2$ expansion of $\vartheta(b_2)$ coincide seems to be very small.
In fact for $b_i\leq40000$ there are only two such couples. Namely:
\begin{align*}
\vartheta (81)=\{0.\overbrace{10}\}_{81}&\textrm{ and }\vartheta (361)=\{0.\overbrace{10}\}_{361}. \\
\vartheta (343)=\{0.\overbrace{19}\}_{343}&\textrm{ and }\vartheta
(1369)=\{0.\overbrace{19}\}_{1369}.
\end{align*}
\end{rem}

\section{The general case}
If $b$ is not a prime power, there is no equality between the $k$
digits of $Z_b(b^k)$ and the first $k$ digits of the base $b$
expansion of $\vartheta(b)$. As a consequence, $Z_b(b^k)$ presents
certain anomalies in its final digits. For instance:
\begin{align*}
Z_{10}(10^{9})=&24999999\textbf{8},\\
Z_{10}(10^{99})=&249999999.......999999\textbf{80},\\
Z_{10}(10^{999})=&249999999.......99999\textbf{791},\\
Z_{10}(10^{9999})=&249999999.......99999\textbf{7859}.
\end{align*}
Or, in a different base:
\begin{align*}
Z_{6}(6^{5})=&\{2,5,5,5,\textbf{4}\}_{6},\\
Z_{6}(6^{6^{2}-1})=&\{2,5,5,5,5,5,5,5,5,5,5,.......5,5,5,5,5,5,5,5,5,5,5,5,5,5,\textbf{4},\textbf{1}\}_{6},\\
Z_{6}(6^{6^{3}-1})=&\{2,5,5,5,5,5,5,5,5,5,5,.......5,5,5,5,5,5,5,5,5,5,5,5,5,5,\textbf{4},\textbf{1},\textbf{5}\}_{6},\\Z_{6}(6^{6^{4}-1})=&\{2,5,5,5,5,5,5,5,5,5,5,.......5,5,5,5,5,5,5,5,5,5,5,5,5,5,\textbf{4},\textbf{0},\textbf{3},\textbf{4}\}_{6},\\
Z_{6}(6^{6^{5}-1})=&\{2,5,5,5,5,5,5,5,5,5,5,.......5,5,5,5,5,5,5,5,5,5,5,5,5,5,\textbf{4},\textbf{0},\textbf{3},\textbf{3},\textbf{4}\}_{6}.
\end{align*}

If $\vartheta(b)$ is not exact it is clear that any convergent sequence with limit $\vartheta(b)$ will share with this value an increasing number of digits. To prove that this is still true even if $\vartheta(b)$ is exact (like in the previous examples, where $\vartheta(10)=\{0.25\}_{10}$ and $\vartheta(6)=\{0.3\}_{6}$) we need to prove the following result.

\begin{prop}
The sequence $\{\gamma_k\}_{k\geq 1}:=\left\{\frac{Z_b(b^k)}{b^k}\right\}_{k\geq 1}$ is strictly increasing. As a consequence, $\displaystyle{\frac{Z_{b}(b^{k})}{b^{k}}<\lim_{n\rightarrow \infty }\frac{Z_{b}(n)}{n}}$ for every $k>0$.
\end{prop}
\begin{proof}
$\displaystyle{\frac{\gamma_{k+1}}{\gamma_k}=\frac{Z_b(b^{k+1})}{bZ_b(b^k)}>1}$ due to Lemma 2 (1).
\end{proof}

We have already seen in the previous section that if $b$ is a prime power, then the number of digits of the base $b$ expansion of $Z_b(b^k)$ is exactly $k$. Now we will see that this is also true for a general $b$.

\begin{prop}
The number of digits of the base $b$ expansion of $Z_{b}(b^{k})$ is exactly $k$; i. e.:
$$\left\lfloor \log_{b}Z_{b}(b^{k})\right\rfloor +1=k.$$
\end{prop}
\begin{proof}
By Lemma 2 (1), we know that $bZ_{b}(b^{k})<Z_{b}(b^{(k+1)})$. Taking logarithms we have $1+\log_{b}Z_{b}(b^{k})<\log _{b}Z_{b}(b^{k+1})$, which clearly implies that $\left\lfloor \log _{b}Z_{b}(b^{k})\right\rfloor<\left\lfloor \log _{b}Z_{b}(b^{k+1})\right\rfloor$.

Thus, the number of digits of the base $b$ expansion of
$Z_{b}(b^{k+1})$ is greater than that of $Z_{b}(b^{k})$. Since
$\left\lfloor \log _{b}Z_{b}(b)\right\rfloor=0$, it follows that
$1+\left\lfloor\log _{b}Z_{b}(b^{k})\right\rfloor\geq k$.

Let us see now that the equality holds. Assume, on the contrary, that $1+\left\lfloor\log_{b}Z_{b}(b^{k_{0}})\right\rfloor
>k_{0}$ for certain $k_{0}$. Then $1+\left\lfloor\log _{b}Z_{b}(b^{m})\right\rfloor>m$ for every $m\geq k_{0}$. This clearly implies that $\displaystyle{\frac{Z_{b}(b^{m})}{b^{m}}>1}$ for every $m>k_{0}$ and $\displaystyle{\vartheta(b)=\lim_{n\rightarrow\infty}\frac{Z_{b}(b^{m})}{b^{m}}\geq 1}$, which is clearly a contradiction since by definition $\vartheta (b)\leq 1$, the equality only holds for $b=2$ and $\displaystyle{\frac{Z_2(2^m)}{2^m}=\frac{2^m-1}{2^m}<1}.$
\end{proof}

We have seen that in the general case the equality between the $k$
digits of $Z_b(b^k)$ and the first $k$ digits of the base $b$
expansion of $\vartheta(b)$ does not hold. It is then interesting to
study how many digits differ.

To do so, let us introduce some notation:
$$\alpha(b,k)=\lfloor\vartheta(b)b^k\rfloor-Z_b(b^k).$$
$$\eta(b,k)=\lfloor\log_b\alpha(b,k)+1\rfloor.$$
Observe that the number of different digits that we are studying is,
at most, $\eta(b,k)+1$. Now we can give an upper bound for
$\eta(b,k)$.

\begin{teor}
The number of digits of the base $b$ expansion of $\alpha(b,k)$ is
smaller or equal than the number of digits of the base $b$ expansion
of $k$, plus 1; i. e.:
$$\eta(b,k)\leq\lfloor\log_b k +2\rfloor.$$
\end{teor}
\begin{proof}
If $k=1$, then $\lfloor\vartheta(b)b\rfloor-Z_b(b)\leq b$, since
$\vartheta(b)\leq1$ and $Z_b(b)\geq0$. This implies that
$\eta(b,1)=\lfloor\log_b\left(\lfloor\vartheta(b)b\rfloor-Z_b(b)\right)+1\rfloor\leq\lfloor\log_b
b +1\rfloor=2=\lfloor\log_b k +2 \rfloor$, as claimed.

Now, let $k\geq 2$. Put $b=p_1^{r_1}\cdots p_s^{r_s}$ and assume, without loss of generality, that $p_1$ is such that $\displaystyle{\min_{1\leq i\leq s}\frac{1}{r_i(p_i-1)}=\frac{1}{r_1(p_1-1)}}$. In that case $\displaystyle{Z_b(b^k)=\left\lfloor\frac{b^k-\sigma_{p_1}(b^k)}{r_1(p_1-1)}\right\rfloor}$ and $\vartheta(b)=\displaystyle{\frac{1}{r_1(p_1-1)}}$. Also observe that if $\beta=\dfrac{b}{p_1^{r_1}}$, then $\sigma_{p_1}(b^k)=\sigma_{p_1}(\beta^k)$.
Now:
\begin{align*}
\lfloor\vartheta(b)b^k\rfloor-Z_b(b^k)&=\lfloor\vartheta(b)b^k\rfloor-\left\lfloor\frac{b^k-\sigma_{p_1}(b^k)}{r_1(p_1-1)}\right\rfloor\leq\frac{b^k-(b^k-\sigma_{p_1}(b^k))}{r_1(p_1-1)}+1=\\&=1+\frac{\sigma_{p_1}(\beta^k)}{r_1(p_1-1)}\leq
1+\frac{(p_1-1)(1+\lfloor\log_{p_1}\beta^k\rfloor}{r_1(p_1-1)}=\\&=1+\frac{1+\lfloor\log_{p_1}\beta^k\rfloor}{r_1}\leq
2+\lfloor\log_{p_1}\beta^k\rfloor=\lfloor\log_{p_1}
p_1^2\beta^k\rfloor\leq\\ &\leq\lfloor\log_{p_1} b^k\rfloor.
\end{align*}

Consequently
$\eta(b,k)=\left\lfloor\log_b\left(\lfloor\vartheta(b)b^k\rfloor-Z_b(b^k)\right)
+1\right\rfloor\leq\left\lfloor\log_b\left(\lfloor\log_{p_1}
b^k\rfloor\right)\right\rfloor +1\leq \lfloor\log_b (k \log_{p_1}
b)\rfloor +1=\lfloor\log_b k+\log_b\log_{p_1}
b\rfloor+1\leq\lfloor\log_b k\rfloor +2=\lfloor\log_b k +2\rfloor $.
\end{proof}

\begin{rem}
The bound obtained in the previous theorem is the best possible one.
In fact, there exists values of the pair $(b,k)$ such that
$\eta(b,k)=\lfloor\log_b k +2\rfloor$. Namely, if $k=b-1$ and
$b<1000$ the following values:
$$b=120,\ 180,\ 240,\ 336,\ 360,\ 378,\ 420,\ 448,\ 504,\ 560,\
630,\ 672,\ 720,\ 756,\ 840,\ 945$$ satisfy that
$\eta(b,b-1)=2=\lfloor\log_b (b-1) +2\rfloor$.
\end{rem}

\begin{cor}
The number of unequal digits between $Z_b(b^k)$ and the first $k$
digits of the base $b$ expansion of $\vartheta(b)$ is smaller or
equal than the number of digits of $k$ plus 2.
\end{cor}
\begin{proof}
The number of unequal digits is, at most, $\eta(b,k)+1$ which, by the previous theorem, is smaller or equal than $\left(\lfloor\log_b k\rfloor +1\right)+2$.
\end{proof}

\begin{rem}
It is interesting to observe that, as far as the authors have been
able to test computationally, the inequality given in the preceding
corollary is strict. Nevertheless we have not found a proof for this
fact, so it remains a conjecture.
\end{rem}

\begin{conj}
The number of unequal digits between $Z_b(b^k)$ and the first $k$
digits of the base $b$ expansion of $\vartheta(b)$ is smaller or
equal than the number of digits of $k$ plus 1.
\end{conj}

If $b$ is square-free, we can improve the bound given in Theorem 2.

\begin{prop}
Let $b=p_1\cdots p_s$ be a square-free integer ($s\geq 2$). Then $\vartheta(b)b^k-Z_b(b^k)\leq k(s-1)$ for every $k$. As a consequence $\eta(b,k)\leq \lfloor\log_b k+\log_b(s-1)+1\rfloor\leq\lfloor\log_b k +1.21\rfloor$.
\end{prop}
\begin{proof}
We can suppose, without loss of generality, that $p_1$ is the greatest prime in the decomposition of $b$. Then $\displaystyle{\vartheta(b)b^k=\frac{b^k}{p_1-1}}$ and also $Z_b(b^k)=Z_{p_1}(b^k)$.

On the other hand, $(p_2\cdots p_s)^k<p_1^{k(s-1)}$, so we can find
$\epsilon_1,\dots,\epsilon_{k(s-1)-1}$ integers smaller that $p_1$
such that
$$\sum_{i=1}^{k(s-1)-1}(\epsilon_i-1)p_1^i\leq (p_2\dots
p_s)^k<\epsilon_1 p_1+\sum_{i=2}^{k(s-1)-1}(\epsilon_i-1)p_1^i.$$ As
a consequence $Z_{p_1}(p_2^k\dots
p_s^k)=\displaystyle{\sum_{i=1}^{k(s-1)-1}(\epsilon_i-1)\frac{p_1^i-1}{p_1-1}}$
and recall that, by Corollary 1, $Z_{p_1}(b^k)=(p_2\dots
p_s)^kZ_{p_1}(p_1^k)+Z_{p_1}(p_2^k\dots p_s^k)$.

Finally, putting all together:
\begin{align*}
\vartheta(b)b^k-Z_b(b^k)&=\frac{p_2^k\dots
p_s^k-\displaystyle{\sum_{i=1}^{k(s-1)-1}}(\epsilon_i-1)(p_1^i-1)}{p_1-1}\leq\\
&\leq\frac{\epsilon_1p_1+\displaystyle{\sum_{i=2}^{k(s-1)-1}}\epsilon_ip_1^i-1-\displaystyle{\sum_{i=1}^{k(s-1)-1}}(\epsilon_i-1)(p_1^i-1)}{p_1-1}=\\
&=\frac{\displaystyle{\sum_{i=1}^{k(s-1)-1}}(\epsilon_i-1)+(p_1-1)}{p_1-1}\leq
k(s-1).
\end{align*}

To end the proof it is enough to recall the definition of $\eta(b,k)$ and to observe that $\log_b(s-1)<0.21$.
\end{proof}

\begin{rem}
The previous proposition can be refined in some special cases. For
instance:
\begin{enumerate}
\item If $b$ is the product of 2 distinct primes; i.e., if $s=2$ is
the proposition, then $\eta(b,k)\leq\lfloor\log_b k+1\rfloor$.
\item If $b$ is square-free and $k$ is a power of $b$, then:
$$\eta(b,k)=\eta(b,b^m)\leq\lfloor m+1.21\rfloor=m+1=\lfloor\log_b k
+1\rfloor.$$
\end{enumerate}
\end{rem}

This remark motivates this final conjecture, which remains open,
that closes the paper.

\begin{conj}
If $b$ is a square-free integer, then $\eta(b,k)\leq\lfloor\log_b
k+1\rfloor$.
\end{conj}

\bibliography{./refgrauoller}
  \bibliographystyle{plain}
\end{document}